\documentclass[12pt]{amsart}
\usepackage{amscd,amsmath,amssymb,amsfonts} 
\usepackage[cmtip, all]{xy}

\newtheorem{thm}{Theorem}[section]
\newtheorem{prop}[thm]{Proposition}
\newtheorem{lem}[thm]{Lemma}
\newtheorem{cor}[thm]{Corollary}

\theoremstyle{definition}
\newtheorem{exm}[thm]{Example}
\newtheorem{defn}[thm]{Definition}

\theoremstyle{remark}
\newtheorem{remk}[thm]{Remark}
\newtheorem{remks}[thm]{Remarks}
\newtheorem{exms}[thm]{Examples}
\newtheorem{notat}[thm]{Notation}

\numberwithin{equation}{section}

{\hfill$\square$\end{defn}}

{\hfill$\square$\end{remk}}

{\hfill$\square$\end{remks}}

{\hfill$\square$\end{exm}}

{\hfill$\square$\end{exms}}

{\hfill$\square$\end{notat}}

\newcommand{\sM}{{\mathcal M}}

\newcommand{\sO}{{\mathcal O}}

\newcommand{\Z}{{\mathbb Z}}

\newcommand{\surj}{\twoheadrightarrow}
\newcommand{\inj}{\hookrightarrow}

\newcommand{\ds}{{/\kern-3pt/}}

\newcommand{\ov}{\overline}

\begin{document}
\title{Gersten conjecture for Equivariant $K$-theory and Applications}
\author{Amalendu Krishna}


\baselineskip=10pt 
                                              
\begin{abstract}
For a reductive group scheme $G$ over a regular semi-local ring $A$, we prove
the Gersten conjecture for the equivariant $K$-theory.
As a consequence, we show that if $F$ is the field of fractions of $A$, then 
$K^G_0(A) \cong K^G_0(F)$, generalizing the analogous result for a dvr by
Serre \cite{Serre}. We also show the rigidity for the $K$-theory with 
finite coefficients of a henselian local ring in the equivariant setting. 
We use this rigidity theorem to compute the equivariant $K$-theory of
algebraically closed fields. 
\end{abstract}

\maketitle   

\section{Introduction}
The classical Gersten conjecture in the algebraic $K$-theory has had 
tremendous amount of applications in the study of algebraic $K$-theory and 
algebraic cycles on smooth schemes. This conjecture was settled by Quillen 
\cite{Quillen} for regular semi-local rings which are essentially of finite 
type over a field. Now let $A$ be a regular semi-local ring and let $F$ 
denote the field of fractions of $A$. Let $G$ be a connected reductive group 
scheme over $A$. For any ring extension $A \to B$, let $R_B(G)$ denote the 
Grothendieck group of the linear representations of $G$ over the base ring $B$.
The extension of scalars gives a natural map 
\begin{equation}\label{eqn:ext}
j : R_A(G) \to R_F(G)
\end{equation}
and one can now ask if this is an isomorphism. This question was asked
by Grothendieck when $G$ is the general linear group over a discrete
valuation ring or more generally, a dedekind domain, and was affirmatively
answered by Serre ({\sl cf.} \cite[Th{\'e}reme~5]{Serre}).
One of the motivations behind the lookout for such an isomorphism
is that the representation ring of an algebraic group over a field is 
relatively easier to compute and one can use the above to compute such
rings over more general rings and this will have applications in the study
of the equivariant $K$-theory of schemes with group actions.
In this note, we show that the isomorphism of ~\ref{eqn:ext} is a direct
consequence of the more general equivariant Gersten conjecture which we now
state. 

Assume that $A$ is a regular semi-local ring which is essentially of finite 
type over a field $k$. Let $F$ be the field of fractions of $A$. Let $G$ be
a connected and reductive affine group scheme over $A$. 
Recall that such a group scheme is said to be {\sl split} if there is a 
maximal torus of $G$ which is defined and split over $A$. For any ring
extension $A \to B$, let $K^G(B)$ (resp. $G^G(B)$) denote the spectrum of
the $K$-theory of finitely generated projective $B$-modules
(resp., finitely generated $B$-modules) with an action of the group scheme
$G_B$ over $B$. For $i \in \Z$, let $K^G_i(B)$ (resp. $G^G_i(B)$) denote
the homotopy groups of the spectrum $K^G(B)$ (resp. $G^G(B)$).
For a prime ideal $\mathfrak p$ of $A$, let $k(\mathfrak p)$ denote
the residue field of $\mathfrak p$.
\begin{thm}\label{thm:Gersten}
For a split and connected reductive group scheme $G$ over the regular
semi-local ring $A$ as above, and for any $i \in \Z$, there is an exact 
sequence
\begin{equation}\label{eqn:Gersten0}
0 \to K^G_i(A) \xrightarrow{i}  K^G_i(F) \xrightarrow{d_1}
\stackrel{}{\underset {{\rm height} \ {\mathfrak p} = 1}{\coprod}}
{K^G_{i-1}\left(k(\mathfrak p)\right)} \to \cdots . 
\end{equation}
\end{thm}
\begin{cor}\label{cor:Gersten1}
The natural map $j : R_A(G) \to R_F(G)$ of representation rings is an
isomorphism. In particular, $R_A(G)$ is noetherian.
\end{cor} 
\begin{cor}\label{cor:Gersten2}
Let $G$ be a split and connected reductive group scheme over the regular
semi-local ring $A$ as above, and let $H$ be a subgroup scheme of $G$ of the
same type. Then the natural map of rings $R_A(G) \to R_A(H)$ is finite.
\end{cor} 

We next turn our attention to the equivariant $K$-theory of henselian
local rings with finite coefficients. Let $A$ be a henselian regular local 
ring over a field $k$, and let $L$ denote the residue field of $A$.
It was shown by Gillet and Thomason ({\sl cf.} \cite[Theorem~A]{GT},
see also \cite{Suslin0} and \cite{Gabber}) that the algebraic $K$-theory
of $A$ with finite coefficients agrees with that of $L$. This was later
used by Suslin ({\sl cf.} \cite{Suslin0}) in a crucial way to compute the 
algebraic $K$-theory of algebraically closed fields, which allowed him to 
settle a conjecture of Quillen. We prove here a similar rigidity theorem in 
the equivariant setting.
\begin{thm}\label{thm:rigid}
Let $A$ be the strict henselization of a ring which is either a 
discrete valuation ring or the local ring at a smooth point
of a variety over a field $k$, and let $L$ denote the residue field of
$A$. Let $G$ be a connected and reductive group scheme over $A$. Then for
any positive integer $n$ prime to the characteristic of $L$, the natural
map
\begin{equation}\label{eqn:rigid0}
K^G_*\left(A, {\Z}/n \right) \to K^G_*\left(L, {\Z}/n \right)
\end{equation}
is an isomorphism.
\end{thm}
In the special case when $G$ is an algebraic group over a field $k$ and $A$ 
is the henselization of a $k$-rational point of a smooth variety over $k$ 
with a trivial $G$-action, this result was also shown by Ostvaer and Yagunov 
({\sl cf.} \cite{OY}) by a different method. We also remark that we can
replace the strict henselization in the above theorem by the henselization,
if one assumes that $G$ is split.  

We finally prove the following equivariant analogue of the main results
of Suslin in \cite{Suslin1} and \cite{Suslin0}. As a consequence,
we obtain an explicit computation of the equivariant $K$-theory for reductive
groups over all algebraically closed fields.
\begin{thm}\label{thm:integer}
Let $G$ be a split connected and reductive group scheme over $\Z$. Let $L$ be 
an algebraically closed field of characteristic $p > 0$ and let $A = 
W(L)$ denote the ring of Witt vectors over $L$. Let $E$ denote the algebraic
closure of the quotient field $F$ of $A$. Then for any positive integer $n$
prime to $p$, there is a canonical
isomorphism $K^G_*\left(L, {\Z}/n \right) \cong K^G_*\left(E, {\Z}/n \right)$. 
\end{thm}
\begin{cor}\label{cor:integer0}
For a connected reductive group $G$ over an algebraically closed field $k$, 
and any $n$ prime to the characteristic of $k$, 
$K^G_i\left(k, {\Z}/n \right)$ is zero (if $i$ is odd) or isomorphic to 
${R_k(G)}/{n}$ (if $i$ is even). 
\end{cor}
\section{Some Preliminaries}   
We give some preparatory results in this section 
in order to prove our equivariant Gersten conjecture.
We fix a regular semi-local ring $A$ which is essentially of finite type
over a field $k$. Let $F$ denote the field of fractions of $A$. Put
$S = {\rm Spec}(A)$. Let $G$ be an affine group scheme over $A$ and let
$G \xrightarrow{\pi} S$ be the structure map.

Recall from \cite{SGA3} that the group scheme $G$ is said to be 
{\sl reductive} if all the geometric fibers of $\pi$ are connected and
reductive algebraic groups. In particular, $G$ is smooth over $S$.
A group scheme $T$ over $S$ is said to {\sl diagonalizable} if there is
finitely generated abelian group $M$ such that the coordinate ring of
$G$ over $S$ is the group algebra $A[M]$. We also write such group schemes 
as $D_S(M)$. Then $T$ is smooth over $S$ if and only if the order of the 
torsion subgroup of $M$ is prime to all the residue characteristics of $A$. 
The group scheme $G$ is called a {\sl Torus} if it is isomorphic to 
$D_S(M)$ in the fpqc topology on $S$, where $M$ is torsion-free.
A closed subgroup scheme $T$ of $G$ is called a {\sl maximal torus} of $G$, 
if $T$ is a torus and every geometric fiber of $T$ is a maximal torus of      
of the corresponding fiber of $G$. 

We shall say that $G$ is {\sl split} (d{\'e}ploy{\'e}) over $S$ if it has a 
maximal torus $T$ which is of the form $D_S(M)$ for a torsion-free abelian 
group $M$ such that $G$ is defined by the root datum $(G, T, M, R)$ over
$S$. Here $R$ is the set of constant functions from $S$ to $M$
({\sl cf.} \cite[Chapter~XXII]{SGA3}). We also recall that a subgroup
scheme $B$ of $G$ is called a {\sl Borel} if every geometric fiber of
$B$ is a Borel subgroup of the corresponding fiber of $G$. Such a Borel
subgroup scheme is called {\sl split} if its unipotent radical is split.
It is known that a reductive group over a field is split in the above 
sense if it contains a split maximal torus. But this is false over a 
general base. However, the following result will show that this is
indeed the case over semi-local rings.
\begin{prop}\label{prop:split0}  
Let $G \xrightarrow{\pi} S$ be a reductive group scheme as above.
Assume that $G$ contains a split maximal torus over $S$. Then $G$ is
split over $S$. In other words, it is given by a root system.
\end{prop}
\begin{proof} {\sl Cf.} \cite[Chapter~XXII, Proposition~2.2]{SGA3}.
\end{proof}
\begin{cor}\label{cor:split1}
Assume that $A$ is a henselian local ring and $G \xrightarrow{\pi} S$
is a reductive group scheme. Then there is a finite Galois 
extension $S' \to S$ such that $G_{S'}$ is split over $S'$. In particular,   
every reductive group scheme over a strictly henselian local ring is
split in the above sense.
\end{cor}
\begin{proof} Since $G$ is reductive, it has a maximal torus $T$ in the
{\'e}tale topology on $S$ by \cite[XXII, Th{\'e}oreme~1.7]{SGA3}.
Since $A$ is henselian, this $T$ is split over a finite Galois extension
of $S$ by \cite[X, Corollary~4.6]{SGA2}. Hence $G$ is given by a root 
system over a finite Galois extension of $S$ by 
Proposition~\ref{prop:split0}. In particular, if $A$ is strictly henselian, 
such a group scheme must be split over $S$ itself.
\end{proof}
All the affine group schemes in sight will be assumed to be connected and
reductive although most of the results of this paper hold for all smooth and 
affine group schemes if the base field is of characteristic zero.

Let $G \xrightarrow{\pi} S$ be a reductive group scheme as before.
The coordinate ring $A[G]$ of $G$ is a Hopf algebra over $A$ such that
$G$ is the spectrum of $A[G]$. In this case, the category of 
$G$-equivariant finitely generated $A$-modules with the trivial action
of $G$ on $S$ is same as the category of finitely generated 
$A[G]$-comodules and is an abelian category ({\sl cf.} \cite{Serre}).
Recall here that a finitely generated $A[G]$-comodule means an 
$A[G]$-comodule which is finitely generated as an $A$-module.
In the same way, the exact category of $G$-equivariant vector bundles
over $S$ with the trivial $G$-action on $S$ is same as the category of
finitely generated $A[G]$-comodules which are projective (and hence free)
as $A$-modules. We refer to {\sl loc. cit.} for the further details on this
equivalence. Let $G^G(A)$ denote the spectrum of the $K$-theory of the
abelian category of finitely generated $A[G]$-comodules, and let
$K^G(A)$ denote the $K$-theory spectrum of the exact category of
finitely generated $A[G]$-comodules which are projective over $A$.
Then there is a natural map of spectra $K^G(A) \to G^G(A)$.
We shall need the following result repeatedly in this paper.
\begin{prop}\label{prop:resolution}
Let $A$ be a regular semi-local ring (or a field) and let $G$ be either a 
split reductive group or a torus over $S = {\rm Spec}(A)$. Let $X$ be a 
smooth quasi-projective $S$-scheme with a $G$-action. Then the map 
$K^G(X) \to G^G(X)$ is a weak equivalence. In particular, if $G$ is any 
reductive group over a strictly henselian regular local ring, then the map 
$K^G(A) \to G^G(A)$ is a weak equivalence.
\end{prop}
\begin{proof} Since $A$ is regular and $G$ is split reductive or torus, every 
finitely generated $G$-equivariant $A$-module has a $G$-equivariant
resolution by $G$-equivariant vector bundles, as shown in 
\cite[Corollary~2.9]{Thomason0}. In particular, $(G, S, S)$ has
the resolution property in the notation of \cite{Thomason0}.
Now since $X$ is smooth and quasi-projective over $S$,  we conclude
from [{\sl loc. cit.}, Lemma~2.10] that every $G$-equivariant coherent
sheaf on $X$ has a finite $G$-equivariant resolution by
$G$-equivariant vector bundles. In particular, the map
$K^G(X) \to G^G(X)$ is a weak equivalence.
The case of strictly henselian ring now follows from this and
Corollary~\ref{cor:split1}.
\end{proof} 

Let ${\sM}^G(A)$ (or ${\sM}^G(S)$) denote the abelian category of 
$G$-equivariant finitely generated $A$-modules. For $p \ge 0$, let 
${\sM}^G_p(A)$ denote the Serre subcategory of those $G$-equivariant $A$ 
modules which are supported on a closed subscheme of codimension at least $p$ 
on $S$. Recall that since $G$ acts trivially on $S$, every subscheme of $S$ 
is $G$-invariant. The following lemma is now elementary.
\begin{lem}\label{lem:elem}
Let $i : Z \inj S$ be a closed subscheme and let $M$ be a coherent
${\sO}_Z$-module such that $i_*(M) \in {\sM}^G(S)$. Then  
$M \in{\sM}^G(Z)$.
\end{lem}
\begin{proof} This follows easily from the definition of the group
action and the fact that the inverse image of $Z$ under the action map
$G \times S \to S$ is $G \times Z$. We skip the details.
\end{proof}

For $p \ge 0$, let $S_p$ denote the set of all codimension $p$ points of $S$.
From the above lemma, we see following Quillen's techniques 
({\sl cf.} \cite{Quillen}) that there is a 
finite filtration of ${\sM}^G(A)$ by Serre subcategories
\[
{\sM}^G(A) = {\sM}^G_0(A) \supset {\sM}^G_1(A) \supset \cdots 
\]
such that for each $p \ge 0$, one has 
\[
\frac{{\sM}^G_p(A)}{{\sM}^G_{p+1}(A)}
\xrightarrow{\cong}
\stackrel{}{\underset {s \in S_p}{\coprod}}
{\stackrel{}{\underset {n}{\bigcup}} 
{\sM}^G\left({\sO}_{S, s}/{{\mathfrak m}^n_{S, s}}
\right)}.
\]
Moreover, if $M \in {\sM}^G\left({\sO}_{S, s}/{{\mathfrak m}^n_{S, s}}
\right)$, then each ${{\mathfrak m}^i_{S, s}}/{{\mathfrak m}^n_{S, s}} M$ is 
in fact a $G$-equivariant submodule of $M$ for $i \le n$ and hence there is
a finite filtration of $M$ by such subsheaves such that each graded
quotient is a $k(s)$-module and hence is in ${\sM}^G\left(k(s)\right)$ 
by Lemma~\ref{lem:elem}. The Devissage theorem 
({\sl cf.} \cite[Theorem~4]{Quillen}) now implies that the map
$G^G\left(k(s)\right) \to 
K\left({\sM}^G\left({\sO}_{S, s}/{{\mathfrak m}^n_{S, s}}\right)\right)$
is a weak equivalence and hence we have the weak equivalence
\begin{equation}\label{eqn:quillen}
K\left(\frac{{\sM}^G_p(A)}{{\sM}^G_{p+1}(A)}\right) \cong
\stackrel{}{\underset {s \in S_p}{\coprod}}
K^G\left(k(s)\right) \ {\forall} \ p \ge 0.
\end{equation}
The equivariant version of Quillen localization sequence ({\sl cf.} 
\cite{Thomason3})
now gives for each $p \ge 0$, a fibration sequence
\begin{equation}\label{eqn:quillen1} 
K\left({\sM}^G_{p+1}(A)\right) \to K\left({\sM}^G_p(A)\right)
\to \stackrel{}{\underset {s \in S_p}{\coprod}}
K^G\left(k(s)\right).
\end{equation} 
Thus we have shown the existence of Quillen spectral sequence in the 
equivariant setting.
\begin{prop}\label{prop:quillen2}
There is a strongly convergent spectral sequence
\[
E^{pq}_1 = \stackrel{}{\underset {s \in S_p}{\coprod}}
K^G_{-p-q}\left(k(s)\right) \Rightarrow G^G_{-n}(A).
\]
\end{prop}
$\hfill \square$
\\
The following equivariant analogue of \cite[Proposition]{Quillen} 
follows directly from ~\ref{eqn:quillen1} and the spectral sequence
in Proposition~\ref{prop:quillen2}
\begin{cor}\label{cor:quillen3}
The following are equivalent. \\
$(i)$ For all $p \ge 0$, the inclusion 
${\sM}^G_{p+1}(A) \to {\sM}^G_{p}(A)$ induces zero map on the $K$-groups. \\
$(ii)$ For all $i \in \Z$, the sequence 
\[
0 \to K^G_i(A) \xrightarrow{j} K^G_i(F) \to
\stackrel{}{\underset {s \in S_1}{\coprod}}
K^G_{i-1}\left(k(s)\right) \to \cdots
\]
is exact.
\end{cor}
$\hfill \square$
\\ 
\section{Equivariant Gersten Conjecture}
Before we prove our Theorem~\ref{thm:Gersten}, we recall from \cite{Krishna}
that since the group scheme $G$ acts trivially on $S$, there is a natural
exact functor ${\sM}(S) \to {\sM}^G(S)$ via the trivial action. Here ${\sM}(S)$
is the category of coherent $S$-modules. This induces a natural map
$G(S) \xrightarrow{f} G^G(S)$ of spectra. Since $G^G_i(S)$ is an 
$R_A(G)$-module, we get a natural map 
\begin{equation}\label{eqn:tensor}
G_i(S) {\otimes}_{\Z} R_A(G) \to G^G_i(S)
\end{equation}
\[
{\alpha} {\otimes} {\rho} \mapsto \left[{\rho} \cdot f(\alpha)\right].
\]
{\bf{Proof of Theorem~\ref{thm:Gersten}:}}
We first assume that $G = T$ is a split torus over $S$ and put $T = D_S(M)$,
where $M$ is a free abelian group of finite type. In this case, every
$T$-equivariant $A$-module $E$ is canonically identified with an $M$-graded 
$A$-module of finite type ({\sl cf.} \cite[Section~3.4]{Serre}). In
particular, one has a canonical decomposition 
$E = \stackrel{}{\underset {\lambda \in M}{\coprod}} E_{\lambda}$. This shows 
that there is a canonical equivalence 
\begin{equation}\label{eqn:tensor0}
G(S)[M] = \stackrel{}{\underset {\lambda \in M}{\coprod}} {G(S)}_{\lambda} 
\to G^G(S).
\end{equation}
Since $K_*\left(G(S)[M]\right) = G_*(S)[M] = G_*(S) {\otimes}_{\Z}
R_A(G)$, we conclude from ~\ref{eqn:tensor0} that the map in 
~\ref{eqn:tensor} is an isomorphism.
Furthermore, we can use Proposition~\ref{prop:resolution} to replace these
$G$-groups with $K$-groups. By the same reason, we have for any point
$s \in S$, a canonical isomorphism 
$K_i\left(k(s)\right) {\otimes}_{\Z} R_{k(s)}(G) \xrightarrow{\cong} 
K^G_i\left(k(s)\right)$. Since $R_A(G) \cong {\Z}[M] \cong R_{k(s)}(G)$,
we conclude that the equivariant Gersten sequence in the split torus case is 
simply the tensor product of the non-equivariant Gersten sequence with 
${\Z}[M]$. Now the non-equivariant Gersten conjecture together with the
flatness of ${\Z}[M]$ as $\Z$-module imply that the equivariant Gersten
sequence is exact. This proves the case of split torus.

We now prove the general case. In view of Corollary~\ref{cor:quillen3}, 
it suffices to show that the map ${\sM}^G_{p+1}(A) \to {\sM}^G_{p}(A)$ 
induces zero map on the $K$-groups for all $p \ge 0$. We choose a split 
maximal torus $T$ inside $G$. Then we get the following commutative
diagram.
\begin{equation}\label{eqn:GT}
\xymatrix{
K_*\left({\sM}^G_{p+1}(A)\right) \ar[r] \ar[d] &
K_*\left({\sM}^T_{p+1}(A)\right) \ar[d] \\
K_*\left({\sM}^G_{p}(A)\right) \ar[r] & K_*\left({\sM}^T_{p}(A)\right)}
\end{equation}
By the proof of the theorem for the torus case and 
Corollary~\ref{cor:quillen3}, we see that the right vertical map is
zero. Thus we only need to show that for all $p \ge 0$, the restriction map
\begin{equation}\label{eqn:GT0}
K_*\left({\sM}^G_{p}(A)\right) \to K_*\left({\sM}^T_{p}(A)\right)
\end{equation}
is injective. 

By Lemma~\ref{lem:elem} and \cite[5.1]{Quillen}, one has
\[
K_*\left({\sM}^G_{p}(A)\right) = \stackrel{}
{\underset {A \surj A', {\rm codim}_A(A') \ge p}{\varinjlim}}
G^G_*(A').
\]
Since the direct limit is an exact functor, it suffices to show that
for any such quotient $A'$, the natural map
\begin{equation}\label{eqn:GT1}
G^G_*(A') \to G^T_*(A')
\end{equation}
is injective. The proof of the theorem is now completed by 
Lemma~\ref{lem:splitting}.
$\hfill \square$
\\
\\
\begin{lem}\label{lem:splitting}
Let $A$ be a noetherian commutative ring such that either it is regular
or essentially of finite type over a field. Let $G$ be a connected and
split reductive group scheme over $A$ with a split maximal torus $T$.
Then the restriction map $G^G_*(A) \to G^T_*(A)$ is split injective.
\end{lem}
\begin{proof} Put $S = {\rm Spec}(A)$. Since $G$ is split reductive, it is 
given by a root system and has a split Borel subgroup scheme $B$ containing 
$T$ and $G/B$ is a projective $S$-scheme 
({\sl cf.} \cite[Chapter~XXII, Proposition~5.5.1]{SGA3}).
We also have maps 
\begin{equation}\label{eqn:GT2**}
G^G_*(A) \to G^B_*(A) \xrightarrow{\cong} G^G_*(G/B) \to 
G^G_*(A).
\end{equation}
Here the last map is the push-forward map induced by the proper map
$G/B \xrightarrow{f} S$. 
Moreover, under our assumption on $A$, we can use 
\cite[Remark~1.9(d) and Theorem~1.10]{Thomason1} to see that
the second map above is an isomorphism. By the same reason, the map
$G^T_*(A) \xrightarrow{\cong} G^B_*(B/T)$ is also an isomorphism. 

Using the projection formula
for the smooth and proper map $G/B \xrightarrow{f} S$, we see that
for any $\alpha \in G^G_*(A)$, one has $f_* f^*(\alpha) =
\left[f_*\left({\sO}_{G/B}\right)\right] \cdot {\alpha}$. On the other hand, 
one has $f_*\left({\sO}_{G/B}\right) = {\sO}_{S}$ by \cite[13.2]{Kempf}.
This shows that the composite map in ~\ref{eqn:GT2**} is identity.
On the other hand, the isomorphism $G^T_*(A) \xrightarrow{\cong} G^B_*(B/T)$
and the homotopy invariance implies that the restriction map   
$G^B_*(A) \to G^T_*(A)$ is an isomorphism, proving the lemma.
\end{proof}

{\bf{Proof of Corollary~\ref{cor:Gersten1}:}}
Since $R_A(G) = K^G_0(A)$ and similarly for $F$, the corollary follows 
directly by using Theorem~\ref{thm:Gersten} for $i= 0$ and by noting
that $K^G_{-1}\left(k(s)\right) = 0$ as ${\sM}^G\left(k(s)\right)$ is
an abelian category. To show that $R_A(G)$ is noetherian, it suffices
now to show that $R_F(G)$ is so. Now it follows directly from
\cite[Th{\'e}oreme~5]{Serre} that $R_F(G) = {R_F(T)}^W$, where $T$ is a
maximal split Torus of $G$ and $W$ is the Weyl group. Since
$R_F(T)$ is a truncated polynomial ring over $\Z$ and hence noetherian,
we deduce from \cite[Lemma~4.4]{KV} that $R_F(G)$ is noetherian too.
$\hfill \square$
\\ 
\\
{\bf{Proof of Corollary~\ref{cor:Gersten2}:}}
By Corollary~\ref{cor:Gersten1}, we can replace $A$ by $F$, and then 
it is already shown in \cite[Proposition~2.1]{Krishna1}.
$\hfill \square$
\\ 
\section{Rigidity for Equivariant K-theory}
Let $A$ be the strict henselization of a ring which is either a 
discrete valuation ring or the local ring
of a smooth point of a variety over a field. 
Let $L$ denote the residue field of $A$. We fix a positive integer $n$ which 
is prime to the characteristic of $L$. Let $G$ be a connected and reductive 
group scheme over $A$. 
By Corollary~\ref{cor:split1}, $G$ is split and hence has a split maximal
torus $T$. Moreover, we have seen before that $G$ contains a split Borel
subgroup scheme $B$ containing $T$. We have seen in  the proof of
Lemma~\ref{lem:splitting} that there are maps  
\begin{equation}\label{eqn:GT2}
G^G_*(A) \to G^B_*(A) \xrightarrow{\cong} G^G_*(G/B) \to 
G^G_*(A).
\end{equation}
such that the composite is identity. Put $X = G/B$ and let
$X \xrightarrow{f} S$ be the smooth and proper map as before.
Note then that $G$ naturally acts on $X$ by left multiplication
and $X$ is a homogeneous $G$-space. Let $i : {\rm Spec}(L) \inj S$ be the 
inclusion of the closed point, and let $X_L$ denote the closed fiber of
$X$.
\begin{lem}\label{lem:commute}
The diagram
\[
\xymatrix{
G^G_*\left(X, {\Z}/n \right) \ar[r]^{{\ov i}^*} \ar[d]_{f_*} & 
G^G_*\left(X_L, {\Z}/n \right) \ar[d]^{f^L_*} \\ 
G^G_*\left(A, {\Z}/n \right) \ar[r]_{i^*} &
G^G_*\left(L, {\Z}/n \right)}
\]
is commutative.
\end{lem}  
\begin{proof} We consider the following Cartesian diagram.
\begin{equation}\label{eqn:commute0}
\xymatrix{
X_L \ar[r]^{\ov i} \ar[d]_{f^L} & X \ar[d]^{f} \\ 
{\rm Spec}(L) \ar[r]_{i} & S}
\end{equation}
Now $f$ and $f^L$ are  clearly smooth and proper maps.
Since $A$ is the strict henselization of the local ring of the smooth
point of a $k$-variety, we see that $A$ is a regular local ring.
Hence $i$ and $\ov i$ are complete intersection maps. 
In particular, they are of finite tor-dimension and the pull-back maps $i^*$
and ${\ov i}^*$ are defined. Moreover, since $f$ is smooth,
we see that $X$ and ${\rm Spec}(L)$ are Tor-independent over $S$. 
The lemma now follows from the equivariant version of 
\cite[Proposition~2.11]{Quillen}
(see also \cite[Proof of Theorem~3.2]{VV}).
\end{proof}
{\bf{Proof of Theorem~\ref{thm:rigid}:}}
Put $\Lambda = {\Z}/n$ and write ${R_A(G)}/n$ by $R_A\left(G, \Lambda \right)$.
As in the proof of Theorem~\ref{thm:Gersten}, we 
first prove the case 
when $G = T$ is a split torus. In this case, we have seen before that
$K^G_*(A) \cong K_*(A)[M]$. Using this and the natural short exact
sequence
\begin{equation}\label{eqn:finite}
0 \to {K^G_i(A)}/n \to K^G_i\left(A, \Lambda \right) \to
{\rm Tor}^1_{\Z} \left(K^G_{i-1}(A), \Lambda \right) \to 0
\end{equation}
({\sl cf.} \cite[Lemma~6.3]{Krishna}), we see that the natural map
$K_i\left(A, \Lambda \right) {\otimes}_{\Lambda} R_A\left(G, \Lambda \right)
\to K^G_i \left(A, \Lambda \right)$ is an isomorphism of 
$R_A\left(G, \Lambda \right)$-modules. The same conclusion also holds for
the $K$-theory of $L$. Now the torus case of the theorem follows from
the non-equivariant rigidity 
({\sl cf.} \cite[Corollary~2.5, Corollary~3.9]{Suslin0}) plus
the canonical isomorphism $R_A\left(G, \Lambda \right) \cong {\Lambda}[M]
\cong R_L\left(G, \Lambda \right)$.

We now prove the general case. By Proposition~\ref{prop:resolution}, we can
replace $K$-theory by $G$-theory. Since $A$ is strictly henselian,
the group scheme $G$ has a split maximal torus $T$ and a Borel subgroup
scheme $B$ containing $T$. We put $X = G/B$ and follow the notations
of Lemma~\ref{lem:commute}. We consider the following diagram.

\[
\xymatrix{
G^G_*\left(A, {\Z}/n \right) \ar[r] \ar[d] &
G^T_*\left(A, {\Z}/n \right) \ar[r] \ar[d] &
G^G_*\left(A, {\Z}/n \right) \ar[d] \\
G^G_*\left(L, {\Z}/n \right) \ar[r] & 
G^T_*\left(L, {\Z}/n \right) \ar[r] &
G^G_*\left(L, {\Z}/n \right)}
\]
The left square is clearly commutative since it is just the pull-back
diagram. The right square commutes by  Lemma~\ref{lem:commute}.
We can now combine the exact sequence ~\ref{eqn:finite} and 
Lemma~\ref{lem:splitting} to see that both the composite horizontal
maps in the above diagram are identity. In particular, the left vertical
map is a retract of the middle vertical map. On the other hand,
we have just shown that the middle vertical map is an isomorphism.
We conclude that the left (and the right) vertical map is also an
isomorphism.
$\hfill \square$
\\

We prove Theorem~\ref{thm:integer} along the lines of the Suslin's
proof of a similar result in the non-equivariant setting.
As such, we begin with the following.
\begin{lem}\label{lem:integer*}
Let $F$ be a henselian discretely valued field with the valuation ring
$A$ and residue field $L$ of positive characteristic $p$. Let $G$ be a 
split reductive group scheme over $A$. For any $i \ge 0$, there is a
short exact sequence
\[
0 \to K^G_i \left(L, \Lambda \right) \to K^G_i \left(F, \Lambda \right)
\xrightarrow{d} K^G_{i-1} \left(L, \Lambda \right) \to 0.
\]
\end{lem}
\begin{proof}
By Theorem~\ref{thm:rigid}, we have isomorphism 
$K^G_i \left(A, \Lambda \right) \cong K^G_i \left(L, \Lambda \right)$. 
We note here
that the statement of Theorem~\ref{thm:rigid} requires $A$ to strictly
henselian. However, as we have remarked before, the strictness was used
only to ensure that $G$ is split.
Now using the equivariant localization sequence, we only need to show that
the map $K^G_i \left(A, \Lambda \right) \to K^G_i \left(F, \Lambda \right)$ 
is injective. We can replace $K$-groups by $G$-groups using
Proposition~\ref{prop:resolution}.

If $G$ is a split torus, we have shown the isomorphism
$G_i \left(A, \Lambda \right) {\otimes}_{\Lambda} R_A \left(G, \Lambda \right)
\cong G^G_i \left(A, \Lambda \right)$. Now the claim follows from
Corollary~\ref{cor:Gersten1} and the 
fact that $G_i\left(A, \Lambda \right) \to G_i\left(F, \Lambda \right)$
is split injective as shown by Suslin 
({\sl cf.} \cite[Corollary~3.11]{Suslin0}).
If $G$ is split reductive with a split maximal torus $T$, we have seen above 
that the map $G^G_i\left(A, \Lambda \right) \to 
G^T_i\left(A, \Lambda \right)$ is split injective. The lemma now follows
from the torus case.
\end{proof}
{\bf{Proof of Theorem~\ref{thm:integer}:}}
If ${E'}/F$ is a finite subextension of $E/F$, then $E'$ is a complete 
discretely valued field with residue field $L$. Since $G$ is split over
$\Z$, it is so over the valuation ring of $E'$. Now we apply
Lemma~\ref{lem:integer*} to get a short exact sequence
\[
0 \to K^G_i \left(L, \Lambda \right) \to K^G_i \left(E', \Lambda \right)
\xrightarrow{d} K^G_{i-1} \left(L, \Lambda \right) \to 0.
\]
Moreover, the exact sequence ~\ref{eqn:finite} implies that
$K^G_{i-1} \left(L, \Lambda \right)$ is of a bounded exponent.
Now the proof of Suslin ({\sl cf.} \cite[Proposition~3.12]{Suslin0})
goes through verbatim in the equivariant setting which completes the
proof of the theorem.
$\hfill \square$
\\
\\
{\bf{Proof of Corollary~\ref{cor:integer0}:}}
If $k$ is an algebraically closed field of positive characteristic and 
if $G$ is a connected reductive group over $k$, then $G$ is split
and hence is given by a root system. Chevalley's theorem then implies that 
such a group $G$ can be lifted to a split group scheme over $\Z$.
Now we can use Theorem~\ref{thm:integer} to reduce to the case when
$k$ is of characteristic zero. In that case, we can use
\cite[Theorem~2]{OY} to assume that $k$ is the field of complex numbers.
If $G$ is now a torus, then the corollary follows from the isomorphism
$K_i \left(k, \Lambda \right) {\otimes}_{\Lambda} R_k \left(G, \Lambda \right)
\cong K^G_i \left(k, \Lambda \right)$ and \cite[Corollary~3.13]{Suslin0}.
If $G$ is any connected reductive group, the result follows again
from the corresponding non-equivariant version of Suslin and the
fact that $K^G_*\left(k, {\Z}/n \right)$ is a retract of
$K^T_*\left(k, {\Z}/n \right)$ as shown above.
$\hfill \square$
\\
\\    

School of Mathematics,
Tata Institute Of Fundamental Research, \\
Homi Bhabha Road,
Mumbai, 400005, India. \\ 
{\sl E-mail address :} amal@math.tifr.res.in

\end{document}